\documentclass[12pt]{amsart}
\usepackage{amsmath,amsthm,amssymb,amsfonts,dsfont}
\usepackage{graphicx}
\usepackage{color}
\usepackage[utf8]{inputenc}

\newtheorem{theorem}{Theorem}[section]
\newtheorem{lemma}[theorem]{Lemma}

\newtheorem{proposition}[theorem]{Proposition}

\newtheorem{corollary}[theorem]{Corollary}
\theoremstyle{definition}
\newtheorem{definition}[theorem]{Definition}
\newtheorem{remark}[theorem]{Remark}
\numberwithin{equation}{section}

\def\llll{\longrightarrow}

\newcommand{\N}{{\mathbb {N}}}

\newcommand{\R}{{\mathbb {R}}}

\newcommand{\1}{{\mathds{1}}}

\def\supp{{\rm supp \;}}
\def\com#1{{``#1''}}

\def\sep{{ \ \  }}
\def\sem{{\ \ \ \  }}
\def\seg{{\ \ \ \  \ \  }}

\title[Bishop-Phelps-Bollob{\'a}s  property] {Bishop-Phelps-Bollob{\'a}s  property  for positive \\ operators  when the domain is $L_\infty $}

\author[M.D. Acosta]{Mar\'{\i}a D. Acosta}
\address{Universidad de Granada, Facultad de Ciencias,
	Departamento de An\'{a}lisis Matem\'{a}tico, 18071 Granada, Spain}
\email{dacosta@ugr.es}
\author[M. Soleimani]{Maryam Soleimani-Mourchehkhorti}
\address{School of Mathematics, Institute for Research in Fundamental Sciences (IPM), P.O. Box: 19395-5746, Tehran, Iran}
\email{m-soleimani85@ipm.ir}

\thanks{The  first  author was  supported  by Junta de Andaluc\'{\i}a grant  FQM--185  and also by Spanish MINECO/FEDER grant PGC2018-093794-B-I00.  The second author was   supported by a grant from IPM}


\begin{document}


 {\large

\begin{abstract}
	We prove that the class of  positive operators from   $L_\infty (\mu)$ to  $Y$  has the Bishop-Phelps-Bollob{\'a}s property for any positive measure $\mu$, whenever $Y$ is a uniformly monotone Banach lattice with a weak  unit.  The same result also holds for the  pair $(c_0, Y)$ for any  uniformly monotone Banach lattice  $Y.$  Further we show that  these results are optimal in case that $Y$ is  strictly monotone.
\end{abstract}
\maketitle

   \section{Introduction}

 In 1961 Bishop and Phelps proved that every continuous  linear functional  on a Banach space can be approximated by norm attaining functionals \cite{BP}. Since then a lot of attention has been  devoted to extend  Bishop-Phelps theorem in the setting of   operators on Banach spaces (see \cite{Acs}). On the other hand Bollob{\'a}s  proved a \com{quantitative version} of that result in 1970 \cite{Bol}. Before stating   this result  we  introduce some notation. By $B_X$, $S_X$ and $X^*$ we denote the closed unit ball, the unit sphere and the topological dual of a Banach space $ X ,$  respectively. If $X$ and $Y$ are both real or  complex Banach spaces,
$ L(X, Y )$ denotes the space of (bounded linear) operators from $X$ to $Y,$ endowed with
its usual operator norm.

\vspace{3mm}

{\it Bishop-Phelps-Bollob{\'a}s theorem} (see \cite[Theorem 16.1]{BoDu} or \cite[Corollary 2.4]{CKMMR}). Let $X$ be a Banach space and $0< \varepsilon < 1$. Given $x \in B_X$  and $x^* \in S_{X^*}$ with $\vert 1- x^* (x) \vert < \frac{\varepsilon ^2 }{2}$, there are elements $y \in S_X$ and $y^* \in S_{X^*}$  such that $y^* (y)=1$, $\Vert y-x \Vert < \varepsilon$ and $\Vert y^* - x^* \Vert < \varepsilon $.

\vspace{2mm}

In 2008,   Acosta, Aron, Garc{\'i}a and Maestre  defined the Bishop-Phelps-Bollob{\'a}s property for operators between  Banach spaces \cite{AAGM}. 

\begin{definition}
	[{\cite[Definition 1.1]{AAGM}}]
	\label{def-BPBp}
	Let $X$  and $Y$ be either real or complex Banach spaces. The pair $(X,Y )$ is said to have the \textit{Bishop-Phelps-Bollob{\'a}s property for operators}  if for every $  0 < \varepsilon  < 1 $  there exists $ 0< \eta (\varepsilon) < \varepsilon $ such that for every $S\in S_{L(X,Y)}$, if $x_0 \in S_X$ satisfies
	$ \Vert S (x_0) \Vert > 1 - \eta (\varepsilon)$, then
	there exist an element $u_0 \in S_X$  and an operator $T \in S_{L(X,Y )}$ satisfying the following conditions
	$$
	\Vert T (u_0) \Vert =1, \sem \Vert u_0- x_0 \Vert < \varepsilon \seg \text{and}
	\sem \Vert T-S \Vert < \varepsilon.
	$$
\end{definition}

Since then a number of interesting results  related to this property have been  obtained  (see \cite{Acbc}).

\vspace{3mm}

Very recently in \cite{ASpos},  the authors introduced   the notion of  Bishop-Phelps-Bollob{\'a}s property for positive  operators between two Banach lattices. Let us mention that  the only difference between this property and the previous one is that in the new property  the operators appearing in Definition \ref{def-BPBp} are positive. In the same paper it is   proved that   the pairs  $(c_0 ,  L_1(\nu) )$ and $(L_\infty (\mu)  ,  L_1(\nu) )$  have the  Bishop-Phelps-Bollob{\'a}s property for positive operators   for any positive measures $\mu$ and  $\nu$ (see \cite[Theorems 3.1 and 2.5]{ASpos}).

\vspace{3mm}

In this paper  we show  a far reaching extension of those results.  More precisely we prove that the pair $(c_0, Y)$ has the  Bishop-Phelps-Bollob{\'a}s property for positive operators    whenever $Y$ is a uniformly monotone Banach lattice (see Corollary 3.3).    We also show that the pair   $(L_\infty (\mu) , Y )$  has the  Bishop-Phelps-Bollob{\'a}s property for positive operators   for any positive measure $\mu$  if $Y$ is a uniformly monotone Banach  lattice  with a weak unit (see Corollary 2.6).  Notice that the last assumption is very mild. For instance, separable Banach lattices have a weak unit (see \cite[Lemma 3, p. 367]{Bi}).

We also remark that not every Banach  function space $Y$ satisfies that the pair $(c_0,Y)$  has the Bishop-Phelps-Bollob{\'a}s  property for positive operators (see \cite[Example 3.2]{ASpos}).
 It is worth also to mention that it is not known whether or not the pair $(c_0, \ell_1)$ has the Bishop-Phelps-Bollob{\'a}s  property for operators in the real case. For some partial results  when the domain is $c_0$ see \cite{KiI},\cite{AGKM}, \cite{ADS} and \cite{AD}.  The paper \cite{KLL} contains a  positive  result  for  the pair $(L_\infty (\mu), Y)$, whenever $Y$ is a uniformly convex  Banach space.

 In the last section of the paper we show  under very mild assumptions that anytime that a pair of  Banach lattices $(X,Y) $ has the BPBp for positive operators,  then $Y$ is indeed uniformly monotone, whenever $X$  admits non trivial $M$-summands and $Y$ is strictly monotone  (see  Proposition \ref{prop-ppp}).
  As a consequence,   the geometrical assumption  on the range space $Y$ in  the main results of sections 2 and 3 are optimal  in case that $Y$ is strictly monotone (see Corollary \ref{cor-optimal}).

We remark that throughout this paper we consider only real Banach spaces.

\section{ Bishop-Phelps-Bollob{\'a}s property for positive operators from $L_\infty$ to a uniformly monotone Banach lattice}

We begin  by recalling some definitions   and  the appropriate notion of  Bishop-Phelps-Bollob{\'a}s property  for positive operators.
For the terminology and basic facts related to  Banach lattices see, for instance,  \cite{AbAl}, \cite{Bi} and \cite{LiTz}. 

An \textit{ordered vector space} is a real vector space $X$  equipped with  an order relation $\le$  that is compatible with the algebraic structure of $X$.   An ordered vector space is  a \textit{Riesz space} if every pair of vectors has a least upper bound  and a greatest lower bound.
A norm $\Vert \ \Vert $ on a  Riesz space $X$ is said to be \textit{ a lattice norm } whenever $|x| \leq |y|$ implies $\Vert x\Vert  \le \Vert y \Vert  $. \textit{A normed Riesz space} is a  Riesz space equipped  with a lattice norm.  
 A  \textit{Banach lattice} is a  normed Riesz space whose  norm is complete.
  A positive element  $e$ in a  Banach lattice  $X$ is a \textit{weak  unit}   if  $ x \wedge e =0$ for some $x \in X$ implies  that $x=0.$
 
A Banach lattice  $E$ is {\it uniformly monotone}  if for every
$\varepsilon > 0$ there is $\delta (\varepsilon)> 0$ such that whenever $x \in S_E$, $ y \in E$ and $x,y \ge 0$
the condition $\Vert x +y \Vert \le 1 + \delta (\varepsilon) $ implies that $\Vert y \Vert \le \varepsilon$.  A  Banach   lattice  $X$  is said to be \textit{order continuous}  whenever  $(x_\alpha ) \downarrow 0 $ in $X$ implies $( \Vert x_\alpha \Vert ) \to 0$.

In case that  $(\Omega,   \mu)$ is a 
measure space,  we denote by $L^0(\mu)$  the space of (equivalence classes of
$\mu$-a.e. equal) real valued measurable functions on $\Omega$. We say that a Banach space $X$ is a \textit{Banach function space}
on $(\Omega,  \mu)$ if  $X$ is an ideal in $L^0 (\mu)$ and whenever $x,y\in
X$ and $|x| \le |y|$ a.e., then $\Vert x \Vert \le \Vert y \Vert .$

An operator $T: X \to Y$ between two ordered vector spaces is called \textit{positive} if $x \geq 0$ implies $Tx \geq 0$.

Let us notice that in case that $Y$ is a uniformly monotone  Banach lattice   it follows from the definition that the function $\delta $ satisfies $\delta (t) \le t$ for every positive real $t$.

\begin{definition} [{\cite[Definition 2.2]{ASpos}}]
	\label{def-BPBp-pos}
	Let $X$  and $Y$ be    Banach lattices. The pair $(X,Y )$ is said to have the {\it Bishop-Phelps-Bollob{\'a}s property for  positive operators}   if for every $  0 < \varepsilon  < 1 $  there exists $ 0< \eta (\varepsilon) < \varepsilon $ such that for every $S\in S_{L(X,Y)}$, such that $S \ge 0$,  if $x_0 \in S_X$ satisfies
	$ \Vert S (x_0) \Vert > 1 - \eta (\varepsilon)$, then
	there exist an element $u_0 \in S_X$  and a positive  operator $T \in S_{L(X,Y )}$ satisfying the following conditions
	$$
	\Vert T (u_0) \Vert =1, \sem \Vert u_0- x_0 \Vert < \varepsilon \seg \text{and}
	\sem \Vert T-S \Vert < \varepsilon.
	$$
\end{definition}

\newpage
\begin{remark}
	\label{u0}
	In case that the pair $(X,Y)$ satisfies the previous definition, if the element  $x_0$  is positive,  then the element  $u_0$  can also be chosen  positive.
\end{remark}

\begin{proof}
	Assume that $(X,Y)$ has the BPBp for positive operators and assume that  $S \in S_{L(X,Y)}$ and $x_0 \in S_X$  satisfy the assumptions in  Definition \ref{def-BPBp-pos} and $x_0$ is also positive. So there exists a pair $(T, u_0) \in S_{L(X,Y) } \times S_X$ such that $T$ is positive  and satisfying
	$$
	\Vert T(u_0) \Vert =1, \sem    \Vert  u_0- x_0 \Vert < \varepsilon    \sem \text {and} \sem \Vert T- S \Vert < \varepsilon .
	$$
	We will check that the positive element $\vert u_0 \vert $ also satisfies the desired conditions.
	
	Notice that from triangle inequality we have $\|  |u_0|  - |x_0| \| \leq \|  u_0 - x_0 \| < \varepsilon $, so  since $x_0$ is positive we conclude  $\|  |u_0|  - x_0 \| \leq \|  u_0 - x_0 \| < \varepsilon $. On the other hand since  the operator $T$ is positive,  $| T(u_0)| \leq   T(|u_0|) .$ 
	Hence  $1=\| T(u_0)\| \leq  \| T(|u_0|)\| \leq \|T\| = 1$, so  $\| T(|u_0|)\| = 1.$  Therefore  the element  $|u_0| \in S_X $ satisfies the desired conditions.
\end{proof}

\vskip4mm

Next we show  some technical results that will be     useful later. Throughout the rest of the section, if $(\Omega, \mu) $   is a measure space, we denote  by $\1$  the constant function equal to $1$ on $\Omega$. Since an element $f$  in  $B_{L_\infty (\mu)}$ satisfies that $ \vert f \vert \le  \1$ a.e.,  it is clear that a positive operator from $L_\infty(\mu)$ to any other Banach lattice satisfies the next assertion.

\begin{lemma}
	\label{le-pos}
Let $\mu$  be a  positive measure and  $T$ a positive operator from 	$ L_\infty (\mu)$ to   some Banach lattice $Y.$  Then $\Vert T \Vert = \Vert T(\1)\Vert$.
\end{lemma}

The next result extends \cite[Lemma  2.4] {ASpos}, where the analogous result was proved for $L_1(\mu).$

\begin{lemma}
	\label{le-dis-supp}
	Let  $Y$ be  a uniformly monotone Banach  function space and $0 < \varepsilon < 1$. Assume that $f_1 $ and $f_2$ are positive elements in $Y$ such that
	$$
	\Vert f_1 + f_2 \Vert \le 1 \sem \text{\rm and} \sem  \frac{1}{ 1 + \delta ( \frac{\varepsilon}{3})}   \le 	\Vert f_1 - f_2 \Vert  ,
	$$
where $\delta $ is the function satisfying  the definition of uniform monotonicity for $Y.$ 	Then there are two  positive functions   $h_1$ and $h_2$  in $Y$   with disjoint supports  satisfying that
	$$
	\Vert h_1 + h_2 \Vert  =1 	 \sem \text{\rm and} \sem \Vert h_i- f_i \Vert  <  \varepsilon \sep \text{for}\sep i=1,2  .
	$$
\end{lemma}
\begin{proof} Assume that $Y$ is a Banach function space  on the   measure space $(\Omega, \mu).$  	We consider the partition of $\Omega$ given by the following sets
	$$
C_1= \{ t \in \Omega:   f_2 (t) \le f_1 (t) \} \sem \text{and} \sem C_2= \{ t \in \Omega:   f_1 (t)  <  f_2 (t) \}  .
	$$
Clearly $C_1$ and $C_2$ are measurable sets. The function $h$ given by 	 $h= 2 \bigl( f_1 \chi_{C_2 } +  f_2 \chi_{C_1 } )$  belongs to $Y$ since $Y$ is a Banach  function space  on $(\Omega, \mu)$.  It is also clear that $h$ is positive and it is satisfied that 
\begin{equation}
\label{f1-f2}
\vert f_1 - f_2 \vert + h = f_1 + f_2.
\end{equation}
 
Clearly we have that 
 \begin{align*}
 \Vert f_1 + f_2 \Vert &  \le 1 \le \Bigl( 1 + \delta \Bigl( \frac{\varepsilon}{3} \Bigr )\Bigr) \Vert f_1 - f_2 \Vert     \\
 & =   \Bigl( 1 + \delta \Bigl( \frac{\varepsilon}{3} \Bigr )\Bigr) \Vert \; \vert  f_1 - f_2 \vert\; \Vert  .   
 \end{align*}
So from \eqref{f1-f2}   it follows that 
$$
\Vert h \Vert   \le \frac{\varepsilon}{3}  \bigl\Vert  \vert f_1 -f_2 \vert \bigr\Vert =  \frac{\varepsilon}{3} \Vert  f_1 -f_2 \Vert \le \frac{\varepsilon}{3} \Vert f_1 +f_2 \Vert \le \frac{\varepsilon}{3}.
$$
As a consequence
\begin{equation}
\label{f2-C2-f1-C1-small}
\Vert f_1 \chi_{C_2} \Vert \le  \frac{\varepsilon}{6} \sem \text{and} \sem \Vert f_2 \chi_{C_1} \Vert \le \frac{\varepsilon}{6}.
\end{equation}
 Now define  $g_i = f_i \chi_{ C_i}$ for $i=1,2$. Notice that  $g_1 $ and $g_2$  are positive functions with disjoint supports, belong to $Y$ and also satisfy 

 \begin{equation}
 \label{g1+g2-less-1}
 \Vert g_1 + g_2 \Vert \leq 1.
 \end{equation}
By \eqref{f2-C2-f1-C1-small} the function $g_1$  satisfies 
\begin{equation}
\label{g1-f1-small}
\Vert g_1 - f_1  \Vert =  \Vert f_1 \chi_{C_1}  - f_1 \Vert  =   \Vert f_1 \chi_{C_2}  \Vert   \le  \frac{\varepsilon}{6}.
\end{equation}

By using the same argument we also obtain that  
\begin{equation}
\label{g2-f2-small}
\Vert g_2 - f_2  \Vert   \le \frac{\varepsilon}{6}.
\end{equation}
 
 It is also satisfied that
 \begin{align}
 \label{g1-g2}
 \Vert g_1 + g_2 \Vert &  \ge  \Vert f_1 + f_2 \Vert - \Vert f_1 - g_1 \Vert -  \Vert f_2 - g_2 \Vert  \nonumber \\
 &\ge     \Vert f_1 - f_2 \Vert -  \frac{\varepsilon}{3}   \sem \text{(by \eqref{g1-f1-small} and \eqref{g2-f2-small})}  \nonumber \\
 &\ge \frac{1}{ 1 + \delta ( \frac{ \varepsilon}{3} ) } - \frac{\varepsilon}{3}
 	 \\
 &\ge  \frac{1}{ 1 +  \frac{ \varepsilon}{3}  } - \frac{\varepsilon}{3} > 0 \nonumber.
 \end{align}
For   $i=1,2$ we can  define the function $h_i$ by $h_i= \frac{g_i}{ \Vert g_1 + g_2 \Vert }.$  Clearly $h_1$ and $h_2$ are positive functions in $Y$ with disjoint supports satisfying also that
$\Vert h_1 + h_2 \Vert =1$.

  For $i=1,2$  we also have that

\begin{align}
\label{hi-gi}
\Vert  h_i - g_i \Vert & =  \Bigl \Vert  \frac{ g_i}{ \Vert g_1 + g_2 \Vert } - g_i \Bigr \Vert 
 \nonumber
   \\
& = \frac{ \Vert g_i \Vert }{\Vert g_1 + g_2 \Vert } \bigl \vert 1 - \bigl \Vert g_1 + g_2 \bigr \Vert \; \bigr \vert     
 \nonumber
\\
&\le 1 -   \bigl \Vert g_1 + g_2 \bigr \Vert  \sem \text{(by  \eqref{g1+g2-less-1})}
  \\
&\le  1-  \frac{1}{ 1 + \delta (\frac{\varepsilon}{3})} + \frac{\varepsilon}{3} \sem \text{(by  \eqref{g1-g2})}
 \nonumber
 \\
 & < \frac{2 \varepsilon}{3}.
 \nonumber
\end{align}

 By  using  \eqref{hi-gi}, \eqref{g1-f1-small} and  \eqref{g2-f2-small} we can  estimate the   distance between $h_i$ and $f_i$ for $i=1,2$ as follows
 $$
 \Vert h_i -  f_i \Vert \le  \Vert h_i -  g_i \Vert + \Vert g_i -  f_i \Vert < \frac{2 \varepsilon}{3} +  \frac{ \varepsilon}{6}  < \varepsilon.
 $$
 This  finishes  the proof.
\end{proof}

\begin{theorem}
	\label{teo-BPBp-L-infty-UM-lattice}
	The pair $(L_\infty (\mu), Y) $ has the Bishop-Phelps--Bollob{\'a}s property for positive operators, for any positive measure $\mu$, whenever  $Y$ is a uniformly monotone Banach function space. The function $\eta$  satisfying Definition \ref{def-BPBp-pos} depends only on the modulus of uniform monotonicity of $Y$.

\end{theorem}
%
\begin{proof}
Assume that $(\Omega _1, \mu)$ is a measure space and $Y$ is a Banach function  space  on $(\Omega _2, \nu).$
Let $ 0 < \varepsilon < 1$	 and $\delta $ be the function satisfying  the definition of uniform monotonicity for the Banach function  space $Y.$
Choose a  real number  $\eta$ such that $0 < \eta = \eta (\varepsilon) < \frac{\varepsilon}{18}$  and  satisfying also
\begin{equation}
\label{eta-small}
\frac{1}{ 1 + \delta (\frac{\varepsilon}{ 18})} < \frac{1}{ 1 + \delta (\eta ^2)}  - 3 \eta.
\end{equation}

	Assume that $f_0 \in S_{L_\infty (\mu)} , S\in  S_{ L( L_\infty (\mu), Y)}	$ and  $S$ is a positive operator such  that 
	$$
	\Vert S(f_0) \Vert > \frac{1}{ 1 + \delta (\eta ^2)} .
	$$
	 We can assume without loss of generality that $\vert f_0  \vert \le 1.$
	We define the sets $A,B$ and $C$ given by
	$$
	A=\{ t \in \Omega_1 : -1 \le f_0(t) < -1 + \eta \}, \sep B=\{ t \in \Omega_1 :  1- \eta <  f_0(t) \le 1 \} \sep 
	$$
	and
	$$
	\sep C=\{ t \in \Omega_1 : \vert f_0(t) \vert \le 1 - \eta \}. 
	$$
Clearly $\{A,B,C\}$ is a partition  of $\Omega _1$ into measurable sets. We clearly have that $\vert f_0\vert + \eta \chi_C \in S_{ L_\infty (\mu)}$.    By using that $S$ is a positive operator we have that
\begin{align*}
\Vert S ( \vert f_0 \vert + \eta \chi_C) \Vert &  \le  1     \\
&<  \Vert S (f_0) \Vert  \bigl( 1 + \delta (\eta ^2) \bigr)  \\
&\le   \Vert S ( \vert f_0 \vert ) \Vert  \bigl( 1 + \delta (\eta ^2) \bigr) .
\end{align*}
In view of the uniform monotonicity of $Y$ the previous inequality implies  that
$
\Vert  S (\eta \chi _C) \Vert \le \eta ^2$ and so
\begin{equation}
\label{S-C-small}
\Vert  S ( \chi _C) \Vert \le \eta.
\end{equation}
From the definition of the sets $A$ and $B$ we  have that $\Vert \chi_A  +  f_0 \chi_A  \Vert _\infty \le \eta $ and   $\Vert f_0 \chi_B - \chi_B \Vert _\infty \le \eta $ and  so
\begin{equation}
\label{S-f0A-close-SA}
\Vert   S(\chi_A) +S (  f_0 \chi_A ) \Vert \le \eta \sem \text{and } \sem \Vert  S (  f_0 \chi_B) - S(\chi_B) \Vert \le \ \eta.
\end{equation}
We clearly obtain that
\begin{align}
\label{SA-SB}
\Vert S(\chi_B) - S(\chi _A ) \Vert &  \ge  \Vert S( f_0\chi_B) + S( f_0\chi _A ) \Vert - \Vert S( f_0\chi_B) -S(\chi _B ) \Vert - \Vert   S(\chi _A )  + S(f_0\chi_A) \Vert
\nonumber \\
 &  \ge  \Vert S( f_0) \Vert -  \Vert S( f_0\chi_C)  \Vert -   2 \eta   \sem \text{(by \eqref{S-f0A-close-SA})}
 \nonumber \\
&>   \frac{1}{ 1+ \delta (\eta ^2)}  - \Vert S(\chi_C)  \Vert - 2 \eta
\\
&\ge    \frac{1}{ 1+ \delta (\eta ^2)}  - 3 \eta \sem \text{(by \eqref{S-C-small})}
\nonumber\\
& >   \frac{1}{ 1+ \delta ( \frac{\varepsilon}{18})}.
\nonumber
\end{align}

Since $S $ is a positive operator and   $\Vert S ( \chi_A) +  S(\chi _B) \Vert \le 1$, in view of \eqref{SA-SB} we can apply Lemma \ref{le-dis-supp}. Hence there are  two positive functions  $h_1$ and $h_2$ in $Y$  satisfying the following conditions
\begin{equation}
\label{h1-SA-h2-SB}
	\Vert h_1- S(\chi_A) \Vert  <   \frac{\varepsilon}{6},  \sem   \Vert h_2- S(\chi_B) \Vert   <    \frac{\varepsilon}{6} ,
\end{equation}

\begin{equation}
\label{h-dis-sup-norm-one}
	\supp h_1 \cap  \supp h_2 = \varnothing  \sem \text{and} \sem \Vert h_1+h_2 \Vert  =1.
	\end{equation}
	
	Hence
	\begin{align}
		\label{S-A-out-sup-h1}	
	\Vert  S (\chi_A) \chi _{ \Omega _2 \backslash \supp h_1} \Vert  &  =  \Vert \bigl  (h_1 - S(\chi_A) \bigr)   \chi _{ \Omega _2 \backslash \supp h_1} \Vert 	  
	\nonumber  \\
	&\le     \Vert h_1 -  S (\chi_A)  \Vert  \\
	& <  \frac{ \varepsilon}{6} \sem \text{\rm (by \eqref{h1-SA-h2-SB})}
	\nonumber
	\end{align}
	and 
		\begin{align}
		\label{S-B-out-sup-h2}	
		\Vert  S (\chi_B) \chi _{ \Omega _2 \backslash \supp h_2} \Vert  &  =  \Vert   \bigl( h_2 - S(\chi_B) \bigr)   \chi _{ \Omega _2 \backslash \supp h_2} \Vert 	  
		\nonumber  \\
		&\le     \Vert h_2 -  S (\chi_B)  \Vert  \\
		& <  \frac{ \varepsilon}{6} \sem \text{\rm (by \eqref{h1-SA-h2-SB})}.
		\nonumber
		\end{align}

	Now we define the operator $U: L_\infty (\mu) \llll Y$ as follows
	$$
U(f)= S(f \chi_A) \chi _{\supp h_1} +  S(f \chi_B) \chi _{\supp h_2} \seg (f \in L_\infty (\mu)).
	$$
	Since $Y$ is a Banach  function space  and $S \in  L(L_\infty (\mu),Y),$ $U$  is well defined  and belongs to $ L(L_\infty (\mu),Y).$ The operator  $U$ is positive since $S$ is positive.  
	It also satisfies that 
	\begin{align}
	\label{U-S}
	\Vert U -  S  \Vert  & =  \sup \bigl \{ \Vert S(f \chi _A) \chi _{ \Omega _2 \backslash \supp h_1} +  S(f \chi_B) \chi _{ \Omega _2 \backslash \supp h_2}    + S(f \chi_C) \Vert  : f \in B_{ L_\infty (\mu)} \bigr\} 
	\nonumber\\
	&  \le   \sup \bigl \{   \Vert  S(f \chi_A) \chi _{ \Omega _2 \backslash \supp h_1} \Vert  +  \Vert  S(f \chi_B) \chi _{\Omega _2 \backslash  \supp h_2}   ) \Vert  +  \Vert  S( f\chi_C) \Vert :  f \in B_{ L_\infty (\mu)} \bigr\}  \nonumber\\
	&  \le  \Vert  S(\chi_A) \chi _{ \Omega _2 \backslash \supp h_1} \Vert  +  \Vert  S( \chi_B) \chi _{\Omega _2 \backslash  \supp h_2}   ) \Vert  +  \Vert  S( \chi_C)  \Vert   \\
	& < \frac{\varepsilon}{3} +   \eta  < \frac{\varepsilon}{2} \sem \text{(by   \eqref{S-A-out-sup-h1},  \eqref{S-B-out-sup-h2} and \eqref{S-C-small})}.  
		\nonumber
		\end{align}
Hence
\begin{equation}
\label{U-1}
\vert \Vert U \Vert - 1 \vert < \dfrac{\varepsilon}{2},	
\end{equation}
so $U \ne 0$.	 

Finally we define $T = \frac{U}{ \Vert U \Vert }$.   Since $U$ is a positive operator, $T$ is also positive. Of course $T \in S_{ L(L_\infty (\mu), Y)}$  and also satisfies
\begin{align}
\label{T-S}
\Vert T - S \Vert & \le  \Vert T - U \Vert + \Vert U - S \Vert  
\nonumber   \\
& < \Bigl \Vert  \frac{U} { \Vert U \Vert } - U \Bigr \Vert + \frac{\varepsilon}{2} \sem \text{(by \eqref{U-S})} \\
& = \bigl \vert 1 - \Vert U \Vert \bigr \vert + \frac{\varepsilon}{2} 
\nonumber \\
& < \varepsilon \sem \text{(by \eqref{U-1})}.
\nonumber 
\end{align}

The  function $f_1$ given by $ f_1 = \chi_B - \chi_A + f_0 \chi _C$ belongs  to $S_{ L_\infty (\mu)}$ and satisfies that
\begin{equation}
\label{f1-f0}
\Vert f_1 - f_0 \Vert _\infty  =  \Vert  \chi_B - \chi_A + f_0 \chi _C - f_0 \Vert _\infty \le \eta  < \varepsilon.
\end{equation}
We clearly have that
$$
U(f_1) = U( \chi_B - \chi_A + f_0 \chi _C ) =  S(\chi _B) \chi _{ \sup h_2} -  S (\chi _A) \chi _{ \sup h_1} .
$$
Since $U$ is a positive operator it satisfies that
$$
\Vert U \Vert = \Vert U ( \1) \Vert =  \Vert  S(\chi_A) \chi _{\supp h_1} +  S( \chi_B) \chi _{\supp h_2} \Vert.
$$
For  each $t \in \Omega _2$ we obtain that
\begin{align*}
\bigl\vert \bigl ( U (\1) \bigr) (t) \bigr \vert & =  \bigl\vert \bigl ( S ( \chi_A) \chi _ {\sup h_1} +   S ( \chi_B) \chi _ {\sup h_2}   \bigr) (t) \bigr \vert  \\
& =   \bigl\vert \bigl (- S ( \chi_A) \chi _ {\sup h_1} +   S ( \chi_B) \chi _ {\sup h_2}   \bigr) (t) \bigr \vert       \sem \text{(by \eqref{h-dis-sup-norm-one})} \\
& =  \bigl\vert \bigl ( U(    \chi_B -  \chi_A  + f_0  \chi _ {C}   \bigr) (t) \bigr \vert \vert   \\
& =   \bigl\vert \bigl ( U( f_1 )  \bigr) (t) \bigr \vert .
\end{align*}
Since $Y$ is a Banach  function  space we conclude that 
$$
\Vert U \Vert = \Vert U( \1) \Vert = \Vert U (f_1) \Vert.
$$

 By \eqref{f1-f0} and    \eqref{T-S}, since $T$ attains its norm at $f_1$, the proof is finished
\end{proof}

The previous result was proved in \cite[Theorem 2.5]{ASpos} in case that the range is a $L_1$ space, so Theorem \ref{teo-BPBp-L-infty-UM-lattice} is already a far reaching extension of that result.

Our purpose now is to  obtain  a version of  Theorem \ref{teo-BPBp-L-infty-UM-lattice} for some abstract Banach lattices.  In order to  get this result, we notice that every uniformly monotone Banach lattice is order continuous (see \cite[Theorem 21, p. 371]{Bi} and \cite[Proposition 1.a.8]{LiTz}). It is also known that any order continuous Banach lattice with a weak  unit  is order isometric  to a Banach function space (see \cite[Theorem 1.b.14]{LiTz}).  From Theorem \ref{teo-BPBp-L-infty-UM-lattice}  and the previous argument we deduce the following result.

\begin{corollary}
	\label{cor-BPBp-pos-L-infty-UM-lattice}
	The pair $(L_\infty (\mu), Y) $ has the Bishop-Phelps-Bollob{\'a}s property for positive operators, for any positive measure $\mu$, whenever  $Y$ is a uniformly monotone Banach lattice with a weak  unit. Moreover,  the function $\eta$  satisfying Definition \ref{def-BPBp-pos} depends only on the modulus of uniform monotonicity of $Y$.
\end{corollary}

\section{Bishop-Phelps--Bollob{\'a}s property for positive operators  in case that the domain is $c_0$.}

\vskip3mm

In this section we show parallel results to  Theorem \ref{teo-BPBp-L-infty-UM-lattice} and Corollary \ref{cor-BPBp-pos-L-infty-UM-lattice} in case that  the domain space is $c_0$. We begin with a   technical result whose proof is straightforward.

\vskip6mm

\begin{lemma}
	\label{le-norm-dom-c0}
	Let $Y$ be a Banach lattice  and $T \in L(c_0,Y)$ be a positive operator. Then the following assertions are satisfied:
	\begin{enumerate}
		\item[1)]  $\Vert T \Vert = \sup \{ \Vert T ( \sum _  {k=1} ^n e_i ) \Vert : n \in \N \} =  \sup \{ \Vert T ( \sum _  {k=1} ^n e_i) \Vert : n \ge N \}$, for all positive integer  $ N.$		
		\item[2)] $ \sup \{ \Vert T (x \chi _C ) \Vert :  x \in B_{c_0} \} =  \sup \{ \Vert T (\chi _{C \cap \{ k \in \N : k \le n\}} ) \Vert :  n \in \N, n \ge N \} $ for all $C \subset \N$ and all positive integer $N.$
	\end{enumerate}
	
\end{lemma}

\begin{theorem}
	\label{teo-BPBp-c0-UM}
The pair $(c_0, Y)$ satisfies the  Bishop-Phelps--Bollob{\'a}s property for positive operators, for any  uniformly monotone Banach function space $Y.$  Moreover,  the function $\eta$  satisfying Definition \ref{def-BPBp-pos} depends only on the modulus of uniform monotonicity of $Y$.
\end{theorem}
\begin{proof}
	The proof of this result is similar  to the proof of Theorem \ref{teo-BPBp-L-infty-UM-lattice}. We include the details of the proof for the sake of completeness.

	Assume that  $Y$ is a Banach  function space on the measure space   $(\Omega, \mu).$ 
	
	Let $ 0 < \varepsilon < 1,$	 and  assume that $Y$  satisfies  the definition of uniform monotonicity  with the function $\delta .$ 
	Choose a   positive real number such that $ \eta  < \frac{\varepsilon}{18}$ satisfying also
	\begin{equation}
	\label{eta-small-c0}
	\frac{1}{ 1 + \delta (\frac{\varepsilon}{ 18})} < \frac{1}{ 1 + \delta (\eta ^2)}  - 3 \eta.
	\end{equation}
	
	Assume that $x_0 \in S_{ c_0} , S\in  S_{ L( c_0, Y)}	$ and  $S$ is a positive operator such  that 
	$$
	\Vert S(x_0) \Vert > \frac{1}{ 1 + \delta (\eta ^2)} .
	$$
	Consider the sets  $A,B$ and $C$ defined  by
	$$
	A=\{ k \in \N : -1 \le x_0(k) < -1 + \eta \}, \sep B=\{ k \in \N :  1- \eta <  x_0(k) \le 1 \} \sep 
	$$
	and
	$$
	\sep C=\{ k \in \N : \vert x_0(k) \vert \le 1 - \eta \}. 
	$$
	Clearly $\{A,B,C\}$ is a partition  of $\N$.  Also the subsets $A$ and $B$ are finite sets, so $\chi_A$ and $\chi_B$ belong to $c_0$.  For a positive integer $n$ denote by $C_n = C \cap \{ k \in \N : k \le n\}$.  We clearly have that $\vert x_0\vert + \eta \chi_{C_n} \in S_{ c_0}$ for all $n \in \N.$     Since $S$ is a positive operator for each positive integer $n$  it is satisfied that 
	\begin{align*}
	\Vert S ( \vert x_0 \vert + \eta \chi_{C_n}) \Vert &  \le  1     \\
	&<  \Vert S (x_0) \Vert  \bigl( 1 + \delta (\eta ^2) \bigr)  \\
	&\le   \Vert S ( \vert x_0 \vert ) \Vert  \bigl( 1 + \delta (\eta ^2) \bigr) .
	\end{align*}
	In view of the uniform monotonicity of $Y$ the previous inequality implies  that
	$
	\Vert  S (\eta \chi _{C_n}) \Vert \le \eta ^2$ for all $ n \in \N$ and so
	$$	
	\Vert  S ( \chi _{C_n}) \Vert \le \eta, \sem \forall n \in \N.
	$$
	From Lemma \ref{le-norm-dom-c0} we deduce that 	
	\begin{equation}
	\label{S-C-small-c0}
	\Vert S( x \chi _{C})   		\Vert  \le  \sup \{ \Vert S ( \chi _{C_n})  \Vert  : n \in \N \} \le \eta, \sem \forall x \in B_{c_0}.
	\end{equation}

	From the definition of the sets $A$ and $B$ we also have that $\Vert  \chi_A + x_0 \chi_A \Vert _\infty \le \eta $ and   $\Vert x_0 \chi_B - \chi_B \Vert _\infty \le \eta, $  so
	\begin{equation}
	\label{S-x0A-close-SA}
	\Vert  S (\chi_A)   + S(x_0 \chi_A )) \Vert \le \eta \sem \text{and } \sem \Vert  S (  x_0 \chi_B) - S(\chi_B) \Vert \le \ \eta.
	\end{equation}
	We clearly obtain that
	\begin{align}
	\label{SA-SB-c0}
	\Vert S(\chi_B) - S(\chi _A ) \Vert &  \ge  \Vert S( x_0\chi_B) + S( x_0\chi _A ) \Vert - \Vert S( x_0\chi_B) -S(\chi _B ) \Vert - \Vert S(\chi _A ) +  S( x_0\chi_A)  \Vert
	\nonumber \\
	&\ge    \Vert S(x_0) \Vert   - \Vert S( x_0\chi_C)  \Vert - 2 \eta \sem \text{(by \eqref{S-x0A-close-SA})}
	\nonumber\\
	&>   \frac{1}{ 1+ \delta (\eta ^2)}  - \Vert S( x_0\chi_C)  \Vert - 2 \eta 
	\\
	&\ge    \frac{1}{ 1+ \delta (\eta ^2)}  - 3 \eta \sem \text{(by \eqref{S-C-small-c0})}
	\nonumber\\
	& >   \frac{1}{ 1+ \delta ( \frac{\varepsilon}{18})}.
	\nonumber
	\end{align}
	
	Since $S $ is a positive operator and    $\Vert S ( \chi_A) +  S(\chi _B) \Vert \le 1$, in view of \eqref{SA-SB-c0} we can apply Lemma \ref{le-dis-supp}. Hence there are  two positive functions  $g_1$ and $g_2$ in $Y$  satisfying the following conditions
	\begin{equation}
	\label{g1-SA-g2-SB}
	\Vert g_1- S(\chi_A) \Vert  <   \frac{\varepsilon}{6},  \sem   \Vert g_2- S(\chi_B) \Vert   <    \frac{\varepsilon}{6} ,
	\end{equation}
	\begin{equation}
	\label{g-dis-sup-norm-one}
	\supp g_1 \cap  \supp g_2 = \varnothing  \sem \text{and} \sem \Vert g_1+g_2 \Vert  =1.
	\end{equation}  	
	Hence
	\begin{align}
	\label{S-A-out-sup-g1}	
	\Vert  S (\chi_A) \chi _{ \Omega \backslash \supp g_1} \Vert  &  =  \Vert \bigl  (g_1 - S(\chi_A) \bigr)   \chi _{ \Omega  \backslash \supp g_1} \Vert 	  
	\nonumber  \\
	&\le     \Vert g_1 -  S (\chi_A)  \Vert  \\
	& <  \frac{ \varepsilon}{6}. \sem \text{\rm (by \eqref{g1-SA-g2-SB})}
	\nonumber
	\end{align}
	and 
	\begin{align}
	\label{S-B-out-sup-g2}	
	\Vert  S (\chi_B) \chi _{ \Omega  \backslash \supp g_2} \Vert  &  =  \Vert   \bigl( g_2 - S(\chi_B) \bigr)   \chi _{ \Omega \backslash \supp g_2} \Vert 	  
	\nonumber  \\
	&\le     \Vert g_2 -  S (\chi_B)  \Vert  \\
	& <  \frac{ \varepsilon}{6} \sem \text{\rm (by \eqref{g1-SA-g2-SB})}.
	\nonumber
	\end{align}

	Now we define the operator $R:  c_0 \llll Y$ as follows
	$$
	R(x)= S(x \chi_A) \chi _{\supp g_1} +  S(x \chi_B) \chi _{\supp g_2} \seg (x \in c_0).
	$$
	Since $Y$ is a Banach function  space and $S \in L(c_0,Y)$, $R$ is well defined and belongs to  $L(c_0,Y)$.   The operator $R$ is  positive since $S$ is positive.  By using Lemma \ref{le-norm-dom-c0}, for any element  $x \in B_{ c_0}$ we  have that 
	\begin{align*}
	\Vert  (R-  S)(x)  \Vert  & =  \Vert  S(x \chi _A) \chi _{ \Omega \backslash \supp g_1} +  S(x \chi_B) \chi _{ \Omega  \backslash \supp g_2}    + S(x \chi_C) \Vert   
	\nonumber\\
	&  \le  \Vert  S( \chi_A) \chi _{ \Omega \backslash \supp g_1} \Vert  +  \Vert  S( \chi_B) \chi _{\Omega \backslash  \supp g_2}   ) \Vert  +   \Vert  S( x \chi_{C} )   \Vert 
	\nonumber\\
	& < \frac{\varepsilon}{3} +   \eta  < \frac{\varepsilon}{2} \sem \text{(by   \eqref{S-A-out-sup-g1},  \eqref{S-B-out-sup-g2} and  \eqref{S-C-small-c0})}.  
	\nonumber 
	\end{align*}
	As a consequence
	\begin{equation}
	\label{R-S}
	\Vert R - S \Vert <  \frac{\varepsilon}{2} \sem \text{and } \sem     	\bigl\vert \Vert R \Vert - 1  \bigr \vert < \frac{\varepsilon}{2},	
	\end{equation}
	so $R \ne 0$.	 
	
	We put  $T = \frac{R}{ \Vert R \Vert }$.    The operator  $T$   is  positive  since $R$ is  positive. Of course $T \in S_{ L( c_0, Y)}$. Now we estimate the distance from $T$ to $S$ as follows 
	\begin{align}
	\label{T-S-c0}
	\Vert T - S \Vert & \le  \Vert T - R\Vert + \Vert R - S \Vert  
	\nonumber   \\
	& < \Bigl \Vert  \frac{R} { \Vert R \Vert } - R \Bigr \Vert + \frac{\varepsilon}{2} \sem \text{(by \eqref{R-S})} \\
	& = \bigl \vert 1 - \Vert R \Vert \bigr \vert + \frac{\varepsilon}{2} 
	\nonumber \\
	& < \varepsilon \sem \text{(by \eqref{R-S})}.
	\nonumber 
	\end{align}

	The element  $u_0$ given by $u_0 = \chi_B- \chi_A + x_0 \chi _C $ belongs  to the unit sphere of $c_0$ since $x_0 \in  S_{ c_0}.$  It also  satisfies 
	\begin{equation}
	\label{u0-x0}
	\Vert u_0 - x_0 \Vert \le \eta < \varepsilon.
	\end{equation}
	
	Since $g_1$ and $g_2$ have disjoint supports,  $Y$ is a Banach  function  space and $R$ is a positive operator,  we also have that
	\begin{eqnarray*}
		\Vert R (u_0) \Vert & = \Vert S ( -\chi_A) \chi_{\supp g_1} +  S ( \chi_B) \chi_{\supp g_2} \Vert   \\
		& =  \Vert S ( \chi_A) \chi_{\supp g_1}  + S ( \chi_B) \chi_{\supp g_2} \Vert  \\
		&  = \Vert R (\chi_{A \cup B})\Vert  = \Vert R\Vert  .
	\end{eqnarray*}
	Hence the operator $T$ also attains its norm at $u_0$.  In view of \eqref{T-S-c0} and 
	\eqref{u0-x0} the proof is finished.

\end{proof}

\begin{corollary}
	\label{cor-BPBp-pos-c0-UM-lattice}
	The pair $(c_0, Y) $ has the Bishop-Phelps--Bollob{\'a}s property for positive operators,  for any uniformly monotone Banach lattice  $Y.$  Moreover,  the function $\eta$  satisfying Definition \ref{def-BPBp-pos} depends only on the modulus of uniform monotonicity of $Y$.
\end{corollary}
\begin{proof}
By using the same argument of Corollary  \ref{cor-BPBp-pos-L-infty-UM-lattice}, in view of  Theorem \ref{teo-BPBp-c0-UM},  we obtain the statement  when $Y$ is a uniformly monotone Banach lattice with a weak  unit.
Notice also that in Theorem \ref{teo-BPBp-c0-UM} the  function  $\eta$ appearing in Definition  \ref{def-BPBp-pos} depends only on the modulus of uniform  monotonicity of the Banach function  space on the range. 

Assume now that  $Y$ is a uniformly monotone Banach lattice with modulus of uniform monotonicity $\delta$. Since  $c_0$ is separable, if $T \in  L(c_0,Y)$, the space $T(c_0)$ is also separable. It is well known that  
%
%
the Banach lattice $X$ generated by $T(c_0)$ is a separable Banach lattice of $Y$ and so, it also satisfies the definition of uniform monotonicity with the function $\delta$.   By \cite[Lemma 3, p. 367]{Bi} $X$  has a weak unit. By the previous arguments, the pair $(c_0,X)$ satisfies the  Bishop-Phelps--Bollob{\'a}s property for positive operators  with a function $\eta $ depending  only on $\delta$. As a consequence, the pair $(c_0,Y)$ also has the  Bishop-Phelps--Bollob{\'a}s property for positive operators.
\end{proof}

Kim proved that the pair $(c_0, Y)$ has the Bishop-Phelps-Bollob{\'a}s property for operators when $Y$ is uniformly convex  (see \cite[Corollary 2.6]{KiI}). Notice that the previous result is a version of that result for positive operators.

\newpage


\section{Results are optimal when the range is strictly monotone}

Our intention now is to show that under some mild assumption on the range space,   Corollaries  \ref{cor-BPBp-pos-L-infty-UM-lattice}
and \ref{cor-BPBp-pos-c0-UM-lattice}  are optimal. For this end we  need some preliminary results.

\begin{lemma}
	\label{normpp}
	Let $M, N$ and $Y$ be normed spaces. Assume  that $T \in {L(M \oplus_{\infty} N, Y)} $ and   $ m + n \in S_{M \oplus_{\infty} N}$ satisfy that  $\|n\| < 1$ and  $\|T(m+n)\| = \|T\|.$ Then  $\|T(m)\| = \|T\|$.
\end{lemma}

\begin{proof}
	If we put $t = \|n\|$, then $ 0 \leq t < 1$. In case that  $t = 0$  it is clear that  $T$ attains its norm at  $m$. Otherwise $0 < t <1$, the element $m + \frac{1}{t} n $ belongs to the unit sphere of $ M \oplus_{\infty} N $ and 
	$$
	m + n = (1- t ) m + t \Bigl( m + \frac{1}{t} n \Bigr).
	$$
	Since  $\|T(m+n)\| = \|T\|, $     and the function 
	$x \mapsto  \Vert T(x) \Vert $  is convex on $M \oplus N$,  we conclude that  $\|T(m)\| = \|T\|$.
\end{proof}

Next result is probably known,  but we did not find a reference in the literature and we include it for the sake of completeness.

\begin{proposition}
	\label{pro-char-UM}
	Let $Y$ be a Banach lattice. The following conditions are equivalent.
	
	\begin{enumerate}
		\item[1)]  $Y$ is uniformly monotone.
		
		\item[2)]  For  every $0< \varepsilon <1$, there is $ \eta (\varepsilon) > 0$ satisfying 
		$$
		u \in Y, v \in S_Y,  \sem 0 \le u\le v \sem \text{and} \sem  \|v - u\|  >   1 - \eta (\varepsilon) \sep \Rightarrow \sep \| u\| \leq  \varepsilon .
		$$

		\item[3)]  For  every $0< \varepsilon <1$, there is $ \eta (\varepsilon) > 0$ satisfying 
		$$
		u, v \in Y,  \sem 0 \le u\le v \sem \text{and} \sem  \|v - u\|  >  ( 1 - \eta (\varepsilon) ) \|v\| \sep \Rightarrow \sep \| u\| \leq  \varepsilon  \|v\|.
		$$
		
	\end{enumerate}
	Moreover, if 2) is satisfied, $Y$ is uniformly monotone with  $\delta (\varepsilon)= \eta (\frac{\varepsilon}{2} )$.  In case that  $Y$ is uniformly monotone with $ \delta (\varepsilon)$  conditions 2) and 3) are satisfied  with $\eta (\varepsilon)= \frac{ \delta (\varepsilon)}{ 1 + \delta (\varepsilon)}.$
\end{proposition}

\begin{proof} 
	1) $\Rightarrow$ 2)	
	\newline
	Assume that  $Y$ is  uniformly monotone with  $\delta (\varepsilon) $.  Assume that $0< \varepsilon <1$ and  $u$ and $v$ are  elements  in $Y$  such that 
	$$ 
	\|v\|=1,  \sem  		0 \leq u \leq v \sem     \text{and}  \sem \|v - u\| > 1 - \eta  (\varepsilon )    =      1- \frac{\delta (\varepsilon )}{1 + \delta (\varepsilon )}.
	$$
	Put $x = v- u , y = u$. So $x , y \geq 0$ and 
	\begin{align}
	\|x + y\| =\|v\|=1 &   = (1 + \delta (\varepsilon )) \Bigl( \frac{1}{1 + \delta (\varepsilon )} \Bigr)
	\nonumber   \\
	& = (1 + \delta (\varepsilon ))  \Bigl( 1- \frac{\delta (\varepsilon )}{1 + \delta (\varepsilon )}\Bigr)
	\nonumber \\ 
	&  <  (1 + \delta (\varepsilon )) \|v - u\|
	\nonumber \\
	& = (1 + \delta (\varepsilon )) \|x\|.
	\nonumber 
	\end{align}

	Hence from uniform monotonicity of $Y$ we conclude $\|u\| = \|y\| \leq \varepsilon \Vert x \Vert \le \varepsilon .$
	
	
	2) $\Rightarrow$ 3)	
	\newline
	Trivial
	
	3) $\Rightarrow$ 1)	
	\newline
	Assume  that $0< \varepsilon <1$,  $x \in S_Y,$  $y \in Y$ satisfy that  $x , y \geq 0$ and $\|x + y\| \leq 1 +  \delta (\varepsilon)=  1+ \eta (\frac{\varepsilon}{2} )$. Put $u= y,  v= x+y $. So  $0 \leq u \leq v $ and
	\begin{align}
	\|v - u \| = \|x\| = 1 & >    \Bigl( 1  - \eta \Bigl ( \frac{\varepsilon}{2} \Bigr)  \Bigr) \Bigl( 1 +  \eta \Bigr(\frac{\varepsilon}{2} \Bigr) \Bigr) 
	\nonumber   \\
	& \geq \Bigl( 1 -  \eta \Bigl(\frac{\varepsilon}{2} \Bigr)  \Bigr) \|x +y \|
	\nonumber \\
	& =  \Bigl ( 1  - \eta \Bigl (\frac{\varepsilon}{2} \Bigr)  \Bigr) \|v\|.
	\nonumber 
	\end{align}
	
	Hence from the assumptions we conclude   $ \|y\|=\| u\| \leq  \frac{\varepsilon}{2}  \|v\| \leq \frac{\varepsilon}{2} ( 1+  \eta (\frac{\varepsilon}{2} )) \le\varepsilon.$ So $Y$ is  uniformly monotone with $\delta (\varepsilon)= \eta (\frac{\varepsilon}{2} )$.	
\end{proof}

\begin{proposition}
	\label{prop-ppp}
	Let $X$  be a Banach lattice, $M$ and $N$ be  non zero
	Banach sublattices of $X$ such that $X= M \oplus _\infty N$ and the  canonical projections from $X$ to $M$ and $N$ are positive operators.
	If  $Y$ is a strictly monotone Banach lattice and the pair $(X,Y)$ has the Bishop-Phelps--Bollob{\'a}s property for positive operators   then $Y$ is  uniformly monotone.
\end{proposition}

\begin{proof}
	We will show that $Y$ satisfies condition 2) in Proposition \ref{pro-char-UM}.    Let   $ 0< \varepsilon < 1$. Let take elements  $u$ and $v$ in $Y$  such that 
	$$ 
	0 \leq u \leq v,  \sem   \|v\| = 1  \sem \text{and}  \sem \|v - u\| > 1 - \eta (\varepsilon),
	$$
	where $\eta$ is the function satisfying the definition of BPBp for positive operators for the pair $(X,Y)$.
	
	Since $M $ and $N$ are non  zero Banach sublattices, there are  positive   elements $m_0 \in S_M$ and $n_0 \in S_N$. 
	By using Hahn-Banach theorem for Banach lattices and   positive elements (see \cite[Theorem 39.3]{Zaan}),  there are positive  functionals  $m_{0}^* \in S_{M^*}$ and $n_{0}^* \in S_{N^*}$ such that $ m_{0}^* (m_0)= 1= n_{0}^* (n_0 ) .$ 
	
	Let  $P$ and $Q$ be  the canonical projections from $X$ to $M$ and $N,$ respectively. 
	Consider  the operator $S$ from $X$  to $Y$ given by
	$$ 
	S(x) =  m_{0}^* (P(x)) (v-u) + n_{0}^* (Q(x)) u,  \sem   (x \in X).
	$$
	Since $v - u$ and $u$ are positive elements  in $Y,$ the functionals  $m_{0} ^* $ and $n_{0}^*$  are positive on $M $ and $N,$ respectively, and the  projections $P$ and $Q$  are positive operators on $X$, then  $S$ is a   positive operator from $X$ to $Y$.
	
	By using the assumptions on $X$, the fact that the functionals $m_{0}^* $ and $n_{0}^*$ belong to the unit sphere of $M^*$ and $N^*$ respectively, since $ v-u$ and $u$  are positive elements in $Y,$ for any element $x \in B_X$  we have that
	\begin{eqnarray}
	\label{norm-S}
	\nonumber
	\Vert S(x) \Vert & \le &  \Vert  \Vert P(x) \Vert ( v-u)  +  \Vert Q(x) \Vert \; u \Vert  \\
	& \le  &  \Vert v \Vert =1. 
	\end{eqnarray}
	Since $m_0+n_0 \in S_X$ and  $S(m_0+n_0)= v \in S_Y$, we deduce that  $S \in S_{ L(X,Y)}.$
	Notice also that $\|S (m_0) \| = \|v-u\|  > 1 - \eta (\varepsilon) $.  By using that  the pair  $(X,Y)$ has the BPBp for positive operators and Remark \ref{u0},   there exist a positive operator $T \in S_{L(X,Y)}$ and  a positive element $x_1 \in S_X$ satisfying 
	$$
	\Vert  T (x_1) \Vert =1, \sem    	
	\Vert T-S\Vert   < \varepsilon  \sem \text{and} \sem   \| x_1 - m_0  \|  < \varepsilon.
	$$
	If we  write $m_1= P(x_1) $ and $n_1 = Q(x_1)$, since  $\|n_1\| = \Vert Q(x_1 -m_0  ) \Vert \le \Vert x_1 - m_0 \Vert <  \varepsilon < 1, $ from Lemma \ref{normpp} we conclude that  $ \Vert T(m_1 ) \Vert = 1$.

	Since $T$ and  $P$ are positive operators and $x_1$ is a  positive  element in $X$, we have that
	$$
	1 = \Vert T(m_1 ) \Vert \le   \Vert T(m_1 + n_0)  \Vert \le 1.
	$$
	By using that $Y$ is strictly monotone we obtain that  $T(n_0)=0.$
	As a consequence, 
	$$
	\| u\|  = \| S( n_0)\| =  \| (S-T) (n_0)  \| \le \| S-T\| < \varepsilon.
	$$
	In view of  Proposition \ref{pro-char-UM}  we  proved that $Y$ is uniformly monotone. 
\end{proof}

As a consequence of  Proposition  	\ref{prop-ppp}  and Corollaries
\ref{cor-BPBp-pos-c0-UM-lattice} and  \ref{cor-BPBp-pos-L-infty-UM-lattice}  
we deduce the following result, which is a version of the one obtained by Kim that asserts  that a Banach space $Y$ is uniformly convex  whenever it is strictly convex and the pair $(c_0,Y)$  has the Bishop-Phelps--Bollob{\'a}s property for operators (\cite[Theorem 2.7]{KiI}).

\vspace{3mm}

It is worth to point out that a Banach lattice is strictly monotone whenever it is strictly convex. 
Analogously  uniform convexity implies uniform monotonicity (see  \cite[Theorem 1]{HKM}, for instance).  Notice that the converse results do not hold since  every $L_1(\mu)$ such that $\dim L_1(\mu) > 1$ is uniformly monotone, but it is not strictly convex.

\vspace{10mm}

\begin{corollary} 
\label{cor-optimal}
	
	\begin{enumerate} 
		\item[1)] Let  $Y$ be a   strictly monotone Banach lattice.  Then the pair $(c_0,Y)$ has the BPBp for positive operators  if and only if $Y$ is uniformly monotone.
		
	\item[2)]  Let $\mu $ be a positive measure such that   $\dim L_\infty (\mu) > 1$ and let   $Y$   be a strictly  monotone Banach lattice.  If   the pair $(L_\infty (\mu), Y)$ has the BPBp for positive operators,   then  $Y$ is uniformly monotone. In case that  $Y$ has a weak unit the converse  is also  true.  
		
	\end{enumerate}

\end{corollary}

\bibliographystyle{amsalpha}

\end{document}